\documentclass[12pt]{amsart}
\pdfoutput=1
\usepackage{amssymb, amsmath}
\usepackage{graphicx} 
\usepackage{textcomp}
\usepackage{blkarray}
\usepackage{easybmat}
\usepackage{multirow}
\usepackage{color}

\definecolor{ao(english)}{rgb}{0.0, 0.5, 0.0}

\newtheorem{theorem}{Theorem}

\newtheorem{lemma}[theorem]{Lemma}
\newtheorem{corollary}[theorem]{Corollary}

\theoremstyle{definition}		

\title{On unknotting tunnel systems of satellite chain links}
\author[D. Gir\~ao, J. M. Nogueira and A. Salgueiro]{Darlan Gir\~ao, Jo\~ao M. Nogueira and Ant\'onio Salgueiro}

\begin{document}

\begin{abstract} 
We prove that the tunnel number of a satellite chain link with a number of components higher than or equal to twice the bridge number of the companion is as small as possible among links with the same number of components. We prove this result to be sharp for satellite chain links over a 2-bridge knot. 
\end{abstract}

\maketitle

\dedicatory{ }

\section{Introduction}\label{intro}

 An unknotting tunnel system for a link $L$ in $S^3$ is a collection of properly embedded disjoint arcs $\{t_1, \ldots, t_n\}$ in the exterior of $L$, such that the exterior of $L\cup t_1\cup \cdots \cup t_n$ is a handlebody. The minimal cardinality of an unknotting tunnel system of $L$ is the \textit{tunnel number} of $L$, denoted $t(L)$.\\

The boundary surface of this handlebody defines a Heegaard decomposition of $E(L)$. We recall that a {\em Heegaard decomposition} of  a compact 3-manifold  $M$ is a decomposition of $M$ into two compression bodies $H_1$ and $H_2$ along a surface $F$, which we refer to as \textit{Heegaard surface}. For the context of this paper, we define the \textit{Heegaard genus} of $M$, denoted by $g(M)$, as the minimal genus of a Heegaard surface of $M$ over all Heegaard decompositions of $M$ into a handlebody and a compression body. If $M$ is an exterior of some link $L$ in $S^3$, $E(L)$, we also have $t(L)=g(E(L))-1$. Note that when $M$ is closed or has connected boundary, any Heegaard decomposition of $M$ consists of at least one handlebody; so, in this case, the Heegaard genus of $M$ is the minimal Heegaard surface genus among all Heegaard decompositions of $M$. However, if the boundary of $M$ has more than one component, as a link exterior can have, then a Heegaard decomposition of $M$ might not decompose $M$ into a handlebody and a compression body; so, in this case, the Heegaard genus of $M$, as defined above, might not be the minimal Heegaard surface genus among all Heegaard decompositions of $M$.\\

If one of the compression bodies of a Heegaard decomposition of genus $g$ is a handlebody, we can naturally present the fundamental group $\pi_1(M)$ with $g$ generators: the core of the handlebody defines $g$ generators, and the compressing disks of the compression body give a set of relators. In this case, the rank $r(M)$ of $\pi_1(M)$, referred to as the rank of $M$, which is the minimal number of elements needed to generate $\pi_1(M)$, is at most $g$. Hence, we have $r(M)\leq g(M)$.\\
 
Under this setting, Waldhausen \cite{Wa} asked whether $r(M)$ can be realized geometrically as the genus of a Heegaard decomposition splitting $M$ into one handlebody and a compression body, that is if $r(M)=g(M)$, for every compact 3-manifold $M$. This question became to be known as the rank versus genus conjecture. In \cite{BZ}, Boileau--Zieschang provided the first counterexamples by showing that there are Seifert manifolds where the rank is strictly smaller than the Heegaard genus. Later, Schultens and Weidman \cite{SW} generalized these counterexam\-ples to graph manifolds. Recently, Li \cite{Li} proved that the conjecture also doesn't hold true for hyperbolic $3$-manifolds. As far as we  know, the conjecture remains open for link exteriors in $S^3$. The first author \cite{Gi} proved this conjecture to be true for augmented links. In this paper, we show that this is also the case for ``most'' of \textit{chain links}, which we proceed to define.   

\medskip

A \textit{satellite $n$-chain link} is a link $L$ defined by a sequence of $n\geq 2$ unknotted linked components  where each component bounds a disk such that each of these disks $D$ intersects the previous and the following disk at exactly two arcs, each of which with only one end point in $\partial D$. Note that if two such disks $D$ and $D'$ intersect at an arc, then this arc has one end point in $\partial D$ and the other end point in $\partial D'$. We denote a collection of such disks by $\mathcal{D}$. A regular neighborhood  of $\mathcal{D}$ is a regular neighborhood of a non-trivial knot $K$. We also refer to $L$ as an \textit{$n$-chain link over $K$} and $K$ as its {\em companion}.
When $K$ is the unknot, $L$ is known in the literature simply as an \textit{$n$-chain link} \cite{NR,Ag,KPR}. When $K$ is a non-trivial knot, $L$ is a satellite link with companion $K$ and pattern an $n$-chain link (over the unknot).\\

\begin{figure}[ht]
\includegraphics[scale=.07]{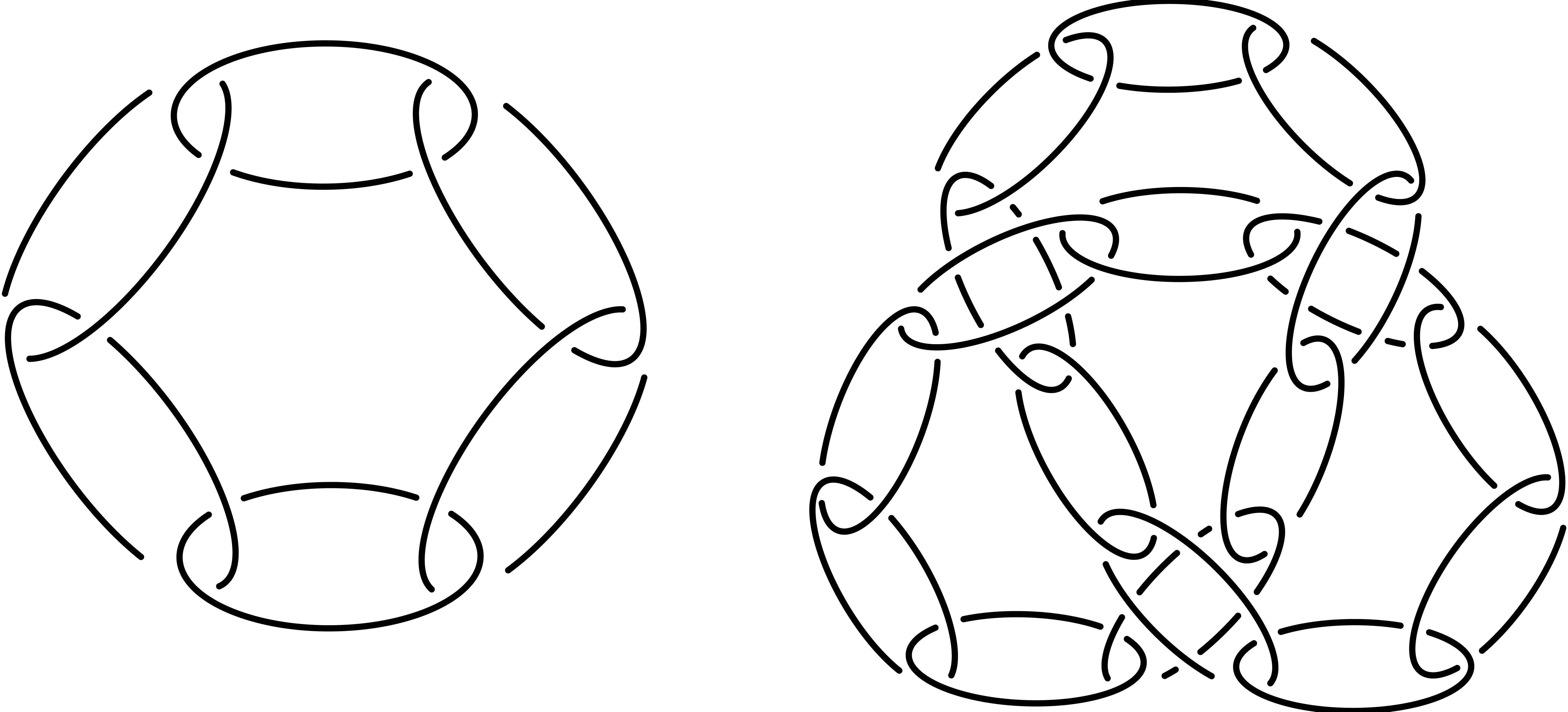}
\caption{Left: a chain link over the unknot; Right: a chain link over the trefoil.}
\label{chains}
\end{figure}

The $n$-chain links over the unknot have been subject of attention for the study of hyperbolic structures. For instance, Neumann and Reid \cite{NR} showed that, for $n\geq 5$, the complement of an $n$-chain link over the unknot admits a hyperbolic structure. Agol \cite{Ag} conjectures that, for $n\leq 10$, an $n$-chain over the unknot is the smallest volume hyperbolic $3$-manifold with $n$ cusps. In \cite{KPR} Kaiser, Purcell and Rollins proved that, for $n\geq 60$, an $n$-chain over the unknot cannot be the smallest volume hyperbolic $3$-manifold with $n$ cusps.\\ 

In this paper we study unknotting tunnel systems of satellite chain links $L$.  We know that if the companion of $L$ is the unknot then the Heegaard genus of $E(L)$ is $n$ (and equal to its rank). In the following theorem we prove that this is also the case for satellite chain links with non-trivial companion as long as the number of components of the link is sufficiently large. We refer to a knot with bridge number $b$ as a $b$-bridge knot. 

\begin{theorem}\label{thm1}
Let $L$ be a $n$-chain link over a $b$-bridge knot $K$.
 If $n\geq 2b$, then the tunnel number of $L$ is $n-1$.
\end{theorem}

An immediate consequence of this theorem is that the rank versus genus conjecture holds true for chain links with sufficiently high number of components: 
Let $L$ be an $n$-chain link over a $b$-bridge knot $K$. If $n\geq 2b$, then $r(E(L))=g(E(L))$. In fact, from Theorem \ref{thm1}, we have $g(E(L))=n$, and from the  ``half lives, half dies" theorem  (\cite{Ha}, Lemma 3.5) applied to $E(L)$, we have $r\big(E(L)\big)\geq n$. Then 
	$n=|L|\leq r\big(E(L)\big)\leq  g\big(E(L)\big)=n$, and $r\big(E(L)\big)=g\big(E(L)\big)=n$.\\

We also prove the following theorem for satellite 3-chain links.

\begin{theorem}\label{thm2}
The tunnel number of a satellite 3-chain link is greater than or equal to 3.
\end{theorem}

Hence, for chain links over 2-bridge knots, Theorem \ref{thm1} is sharp:

\begin{corollary}\label{cor2}
If $L$ is a $n$-chain link over a 2-bridge knot $K$, then the tunnel number of $L$ is $n-1$ if and only if $n\geq 4$.
\end{corollary}

This is a consequence of Theorems  \ref{thm1} and \ref{thm2} and of satellite 2-chain links not having tunnel number one, as proved in \cite{MU} by Eudave-Mu\~noz and Uchida (or by following an argument as in the proof of Theorem \ref{thm2}). The authors wouldn't be surprised Theorem \ref{thm1} to be sharp for any bridge number of $K$.

\bigskip

This paper is organized into two sections, one for the proof of each theorem mentioned above. Throughout the paper we assume all manifolds to be in general position.

\section*{Acknowledgments}
The first author was partially supported by CNPq grants 446307/\linebreak 2014-9 and 306322/2015-3. The second and third authors were partially supported by the Centre for Mathematics of the 
University of Coimbra - UIDB/00324/2020, funded by the Portuguese Government through FCT/MCTES.

We also thank Mario Eudave-Mu\~noz for discussions on the subject of the paper.

\section*{Dedicatory}
While this paper was under preparation, Darlan Gir\~ao was diagnosed with cancer. After a prolonged courageous and dignifying battle with his condition, Darlan died before we could finish this work together. Darlan has been a very good friend and colleague, who we will miss. This paper is dedicated to his memory.

\section{Tunnel number of chain links with large number of components}
\label{section:proof_of_theorem}

In this section we prove Theorem \ref{thm1}.

Let $\mathcal{D}$ be a collection of disks as in the definition of satellite chain link and $\mathcal{A}$ the collection of arcs of intersection between the disks of $\mathcal{D}$. 
Let $R$ be a regular neighborhood of $\mathcal{D}$, such that $R$ is also a regular neighborhood of $K$.
Consider also a sphere defining a $b$-bridge decomposition for $K$, denoted by $S$, intersecting $R$ in a collection of meridional disks. Denote by $B$ and $B'$ the balls bounded by $S$ in $S^3$. 

Since $n\geq 2b$, we can perform an ambient isotopy so that each component of $B\cap R$ contains exactly one arc of $\mathcal{A}$, and each component of $B'\cap R$ contains at least one arc of $\mathcal{A}$. In the exterior of $L$, we start by adding $n-b$ tunnels to $N(L)$, denoted $t_1, \ldots, t_{n-b}$, corresponding to regular neighborhoods of the arcs of $\mathcal{A}$ in $B'\cap R$. (See Figure \ref{figure:balls}.) 

\begin{figure}[ht]
	\includegraphics[scale=.15]{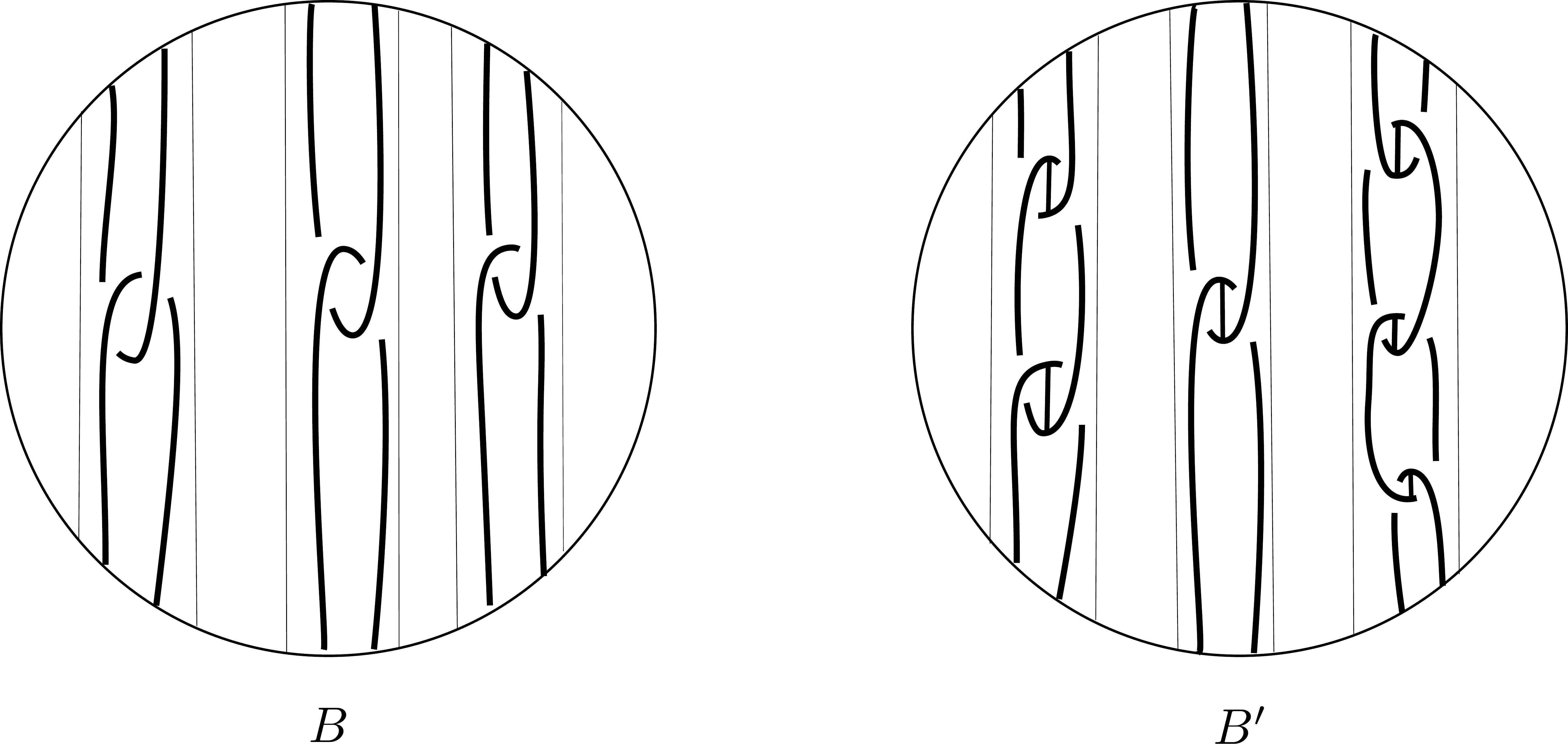}
	\caption{An illustration of $B$ and $B'$, with $n-b$ tunnels in $B'$.}
	\label{figure:balls}
\end{figure}

After an ambient isotopy of $N(L\cup t_1\cup\cdots\cup t_{n-b})$, we obtain in $B'$ a regular neighborhood $N(\Gamma)$ of  a graph $\Gamma$ obtained from the $b$ components of $K\cap B'$, denoted $c_1\cup \cdots\cup c_b$, by adding $n-2b$ arcs parallel to $K$. Note that after the isotopy, $S$ intersects $N(L\cup t_1\cup \cdots\cup t_{n-b})$ in $2b$ disks. (See Figure \ref{figure:graph2}.)

\begin{figure}[ht]
	\includegraphics[scale=.15]{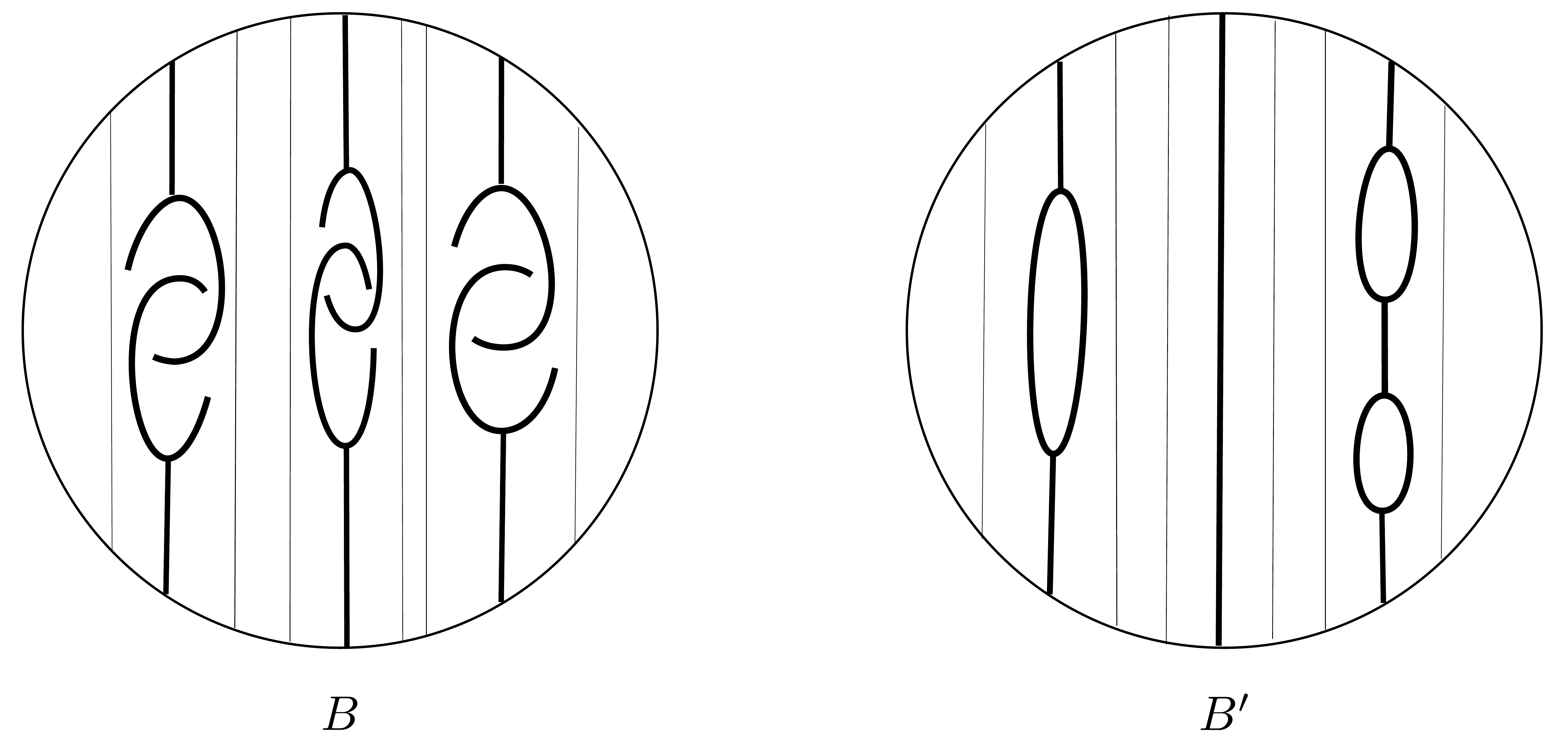}
	\caption{The graph $\Gamma$ in $B$ and $B'$.}
	\label{figure:graph2}
\end{figure}

As $(B'; c_1, \ldots, c_b)$ is a trivial tangle, we add $b-1$ tunnels, denoted $t_{n-b+1},\ldots, t_{n-1}$, to $N(\Gamma)$ in its exterior in $B'$, such that $N(\Gamma\cup t_{n-b+1}\cup \cdots\cup t_{n-1})$ can be isotoped to become the whole $B'$ with $n-2b$ trivial 1-handles. (See Figure \ref{figure:graph3}.)

\begin{figure}[ht]
	\includegraphics[scale=.15]{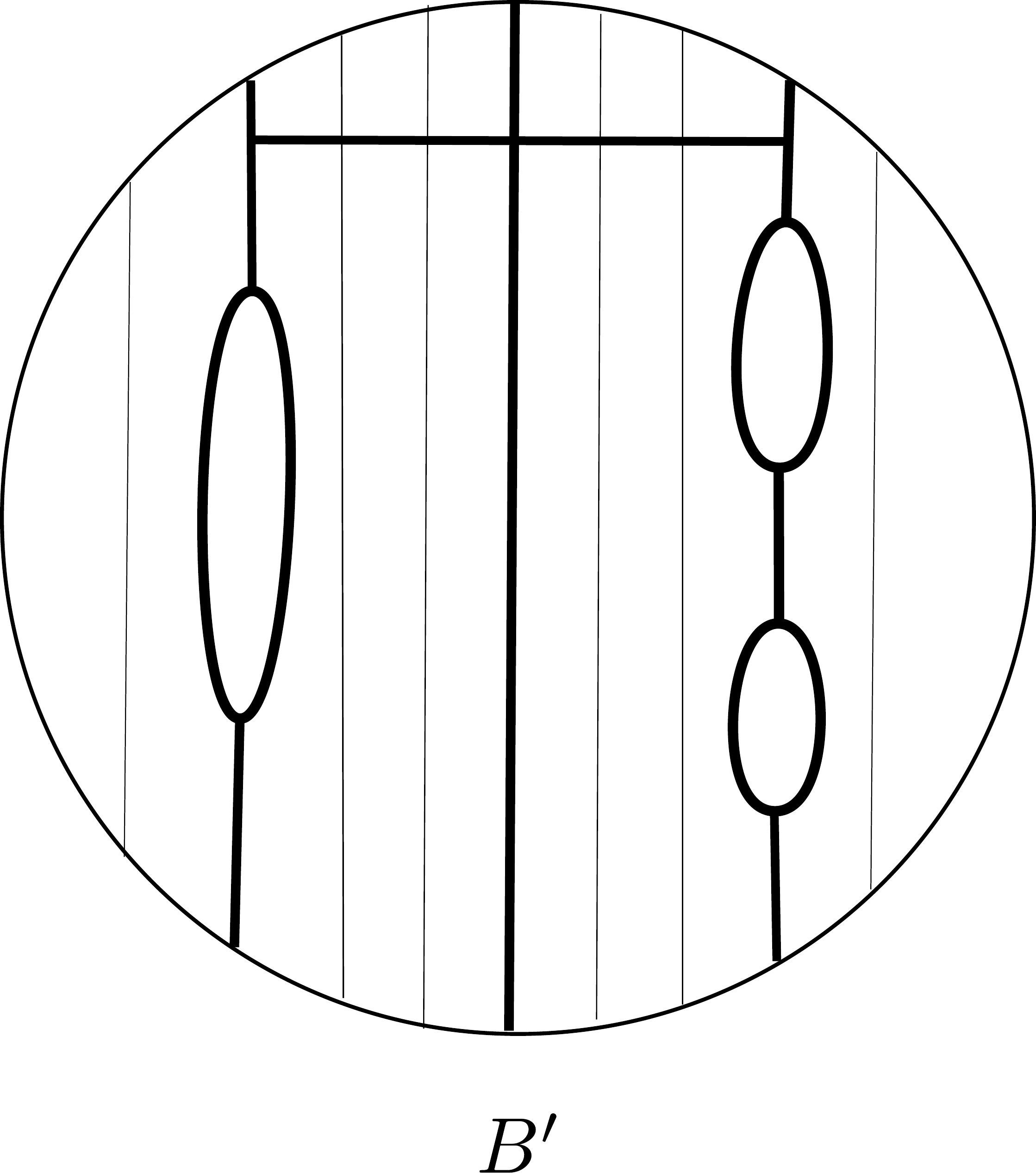}
	\caption{The graph $\Gamma\cup t_{n-b+1}\cup\ldots\cup t_{n-1}$ in $B'$.}
	\label{figure:graph3}
\end{figure}

Hence, the resulting space of the exterior of $L\cup t_1\cup\cdots\cup t_{n-1}$ is ambient isotopic to the exterior in $B$ of the union of $B\cap L$ with $n-2b$ trivial arcs in $B-L$ (that is, each arc co-bounds a disk with $B-L$ and these disks are disjoint). Since those $n-2b$ arcs are trivial, the exterior of $L\cup t_1\cup\cdots\cup t_{n-1}$ is a handlebody if and only if the exterior of $B\cap L$ in $B$  is a handlebody.
The components of $B\cap R$ cobound, each with an arc in $S$, mutually disjoint disks (from the bridge decomposition of $K$ defined by $S$). Let $D_R^i$, for $i=1, \ldots, b$, be a collection of these disks, $B^i_R$, for $i=1, \ldots, b$, the ball corresponding to the regular neighborhood in $B$ of $D_R^i$, together with the corresponding component of $B\cap R$, and let $B_R$ be the exterior in $B$ of the union of these balls. As the intersection of $B^i_R$ and $\partial B$ is a disk, $B_R$ is a ball. The components of $L$ in each 1-handle of $B\cap R$ define a trivial tangle in the respective 1-handle and in $B_R^i$. As $B_R^i \cap B_R$ is a disk disjoint from $L$, we have that the exterior in $B$ of $L$ is obtained by gluing handlebodies along a disk. That is, the exterior of $B\cap L$ in $B$ is a handlebody. Therefore, the tunnel number of $L$ is at most  $n-1$ and, as $L$ has $n$ components, it is also at least $n-1$. Hence, the tunnel number of $L$ is $n-1$. 

\section{Unknotting tunnel systems of satellite 3-chain links}\label{section:rank}

In this section, we will show that the tunnel number of a $3$-chain link $L$ over a non-trivial knot is at least $3$, and, hence, prove Theorem \ref{thm2}. Note that it is at least 2, because $L$ has three components. Suppose, by contradiction, that the tunnel number of $L$ is 2. 

Denote the components of $L$ by $L_i$, for $i=1, 2, 3$, and, respectively, by $D_i$, the components of $\mathcal{D}$, the disks they bound, as in the definition of satellite chain link.  We denote also by $L_i$ a regular neighborhood of the corresponding component of $L$. The regular neighborhood of $D_1\cup D_2\cup D_3$ is a solid torus, with $K$ its core. Let $S$ be a sphere, such that $S-L$ is essential in the exterior of $L$, bounding a ball $B$ intersecting $L$ only in two arcs of the same component of $L$, say $L_1$, and with the other components in the exterior of $B$. Such a sphere exists, since each pair of components of $L$ is linked, and the two arcs of $B\cap L$ have to be parallel with a knotted pattern $K_0$. We refer to $B$ and its complement as the {\em inside} and {\em outside} of $S$, respectively. (See Figure \ref{figure:sphere}).

\begin{figure}[ht]
	\includegraphics[scale=.1]{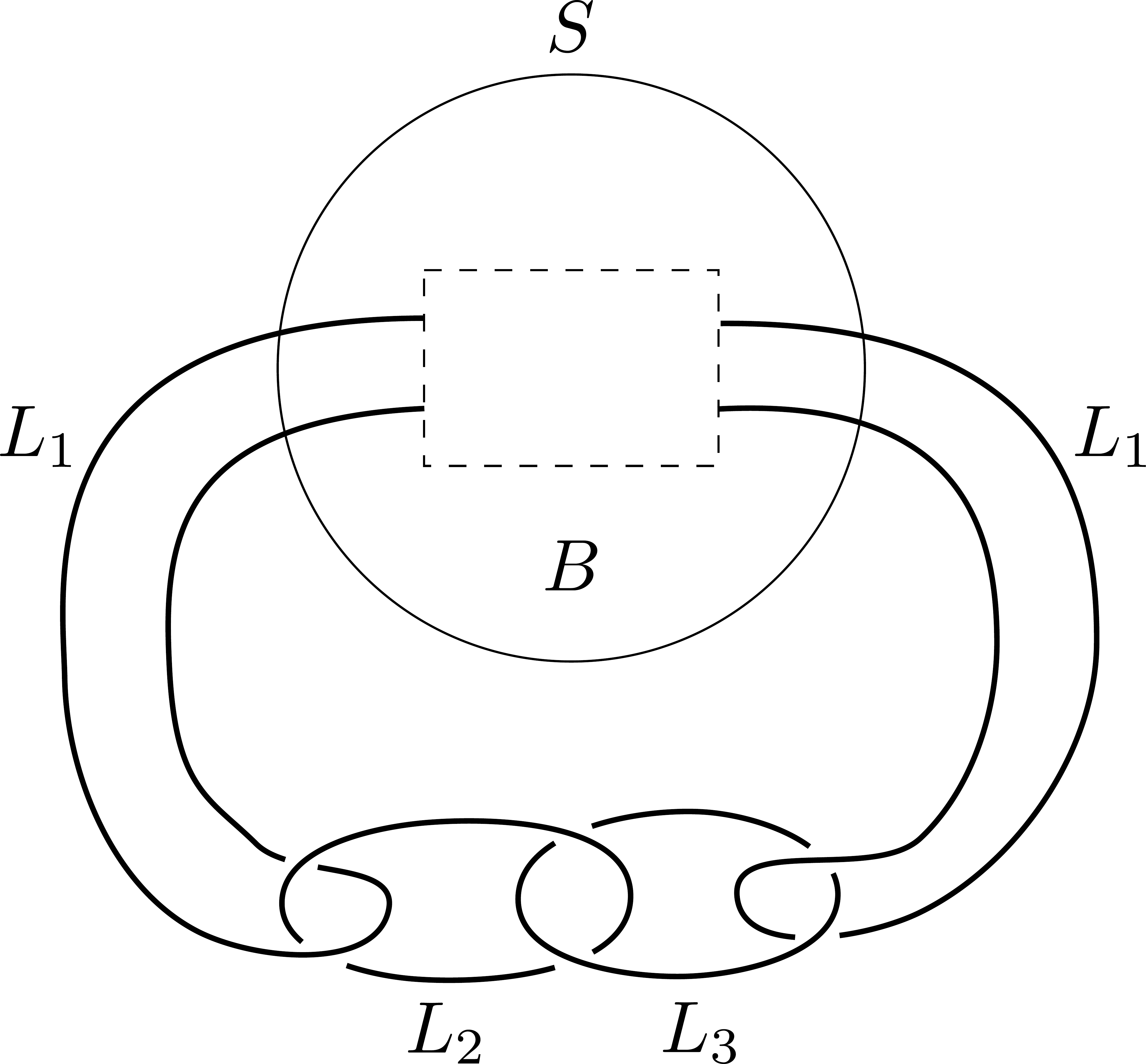}
	\caption{The sphere $S$ intersecting $L_1$ and the components $L_2$ and $L_3$ on the outside of $S$.}
	\label{figure:sphere}
\end{figure}
	
Since the tunnel number of $L$ is $2$, there exists a system $\tau$ of two arcs in the exterior of $L$ such that the exterior $H$ of $L\cup\tau$ is a genus 3 handlebody. We denote a regular neighborhood of $L\cup\tau$ by $G$.  We denote also by $\tau$ a regular neighborhood of these arcs.

\begin{lemma}\label{tau and S}
	No $\tau$ is disjoint from $S$.
\end{lemma}
\begin{proof}
	If there is some $\tau$ disjoint from $S$, then it is  outside $S$, since $\tau$ has two components and $L_2$ and $L_3$ are outside $S$. Hence, there is an essential torus in the handlebody $H$, the torus that follows the pattern of $K_0$ in $B$, which is a contradiction.
\end{proof}

Among all possible $\tau$ and spheres $S$, consider  a pair such that the number of intersections of $\tau$ and $S$, $|\tau\cap S|$, is minimal. By Lemma \ref{tau and S}, $\tau \cap S$ is non-empty. We refer to a disk of intersection of $\tau$ with $S$ as a $t$-disk, and we denote them by $t_1,\ldots,t_n$, and  we refer to a disk of intersection of $L_1$ with $S$ as a $c$-disk, and we denote them by $c_1,c_2,c_3,c_4$. 
The components of $G-S$ are balls, solid tori or a genus two handlebody. Each component of $G-S$ containing $L_2$ or $L_3$ has genus one or two. In case it has genus two, the $c$-disks are all parallel in $G$, otherwise the $c$-disks are parallel two-by-two or three-by-one, as schematically illustrated in Figure \ref{figure:components}. Since $K$ is non-trivial, the punctured sphere $S$ is essential in the exterior of $L$.

\begin{figure}[ht]
	\includegraphics[scale=0.046]{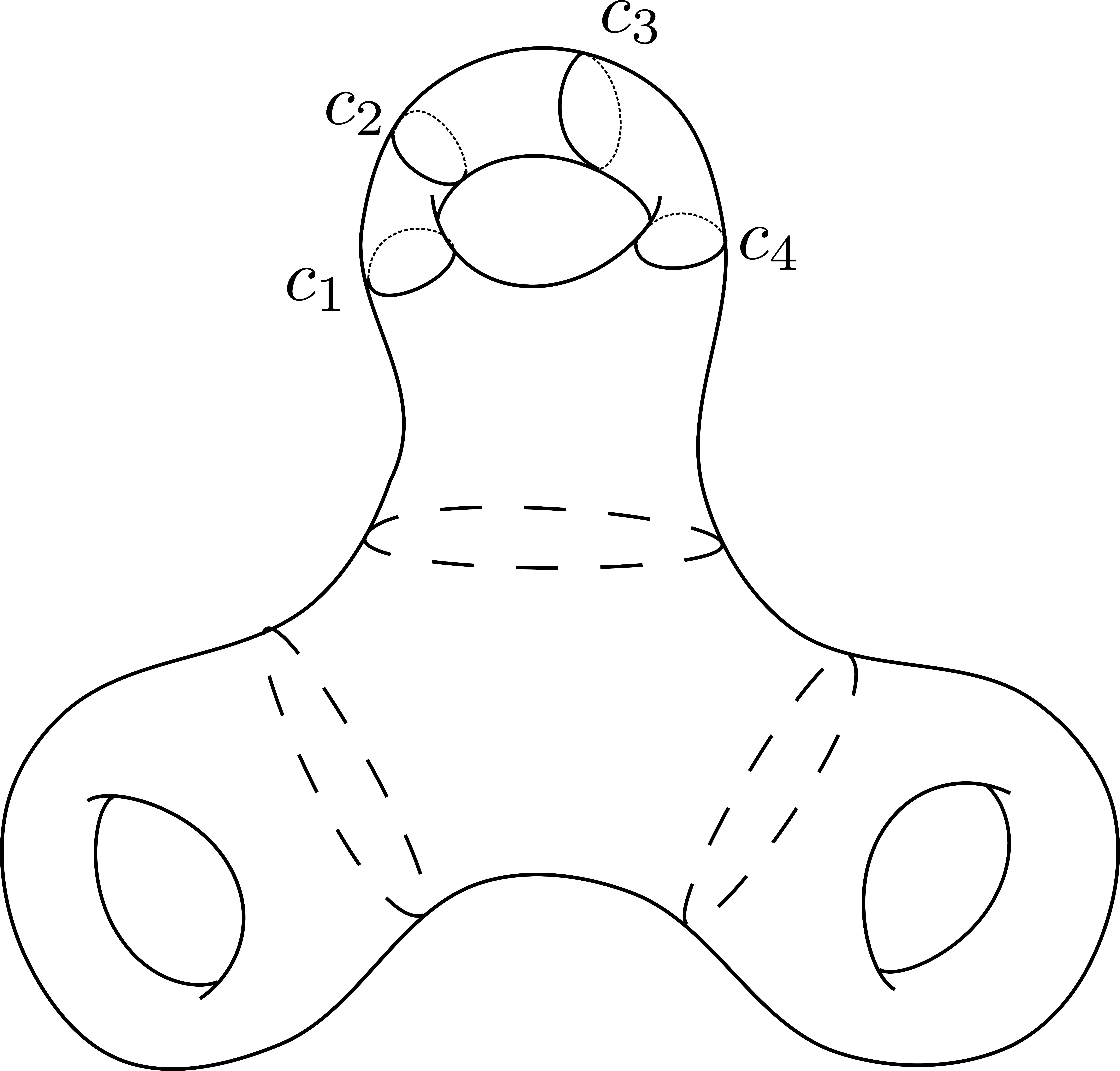}
	\includegraphics[scale=0.046]{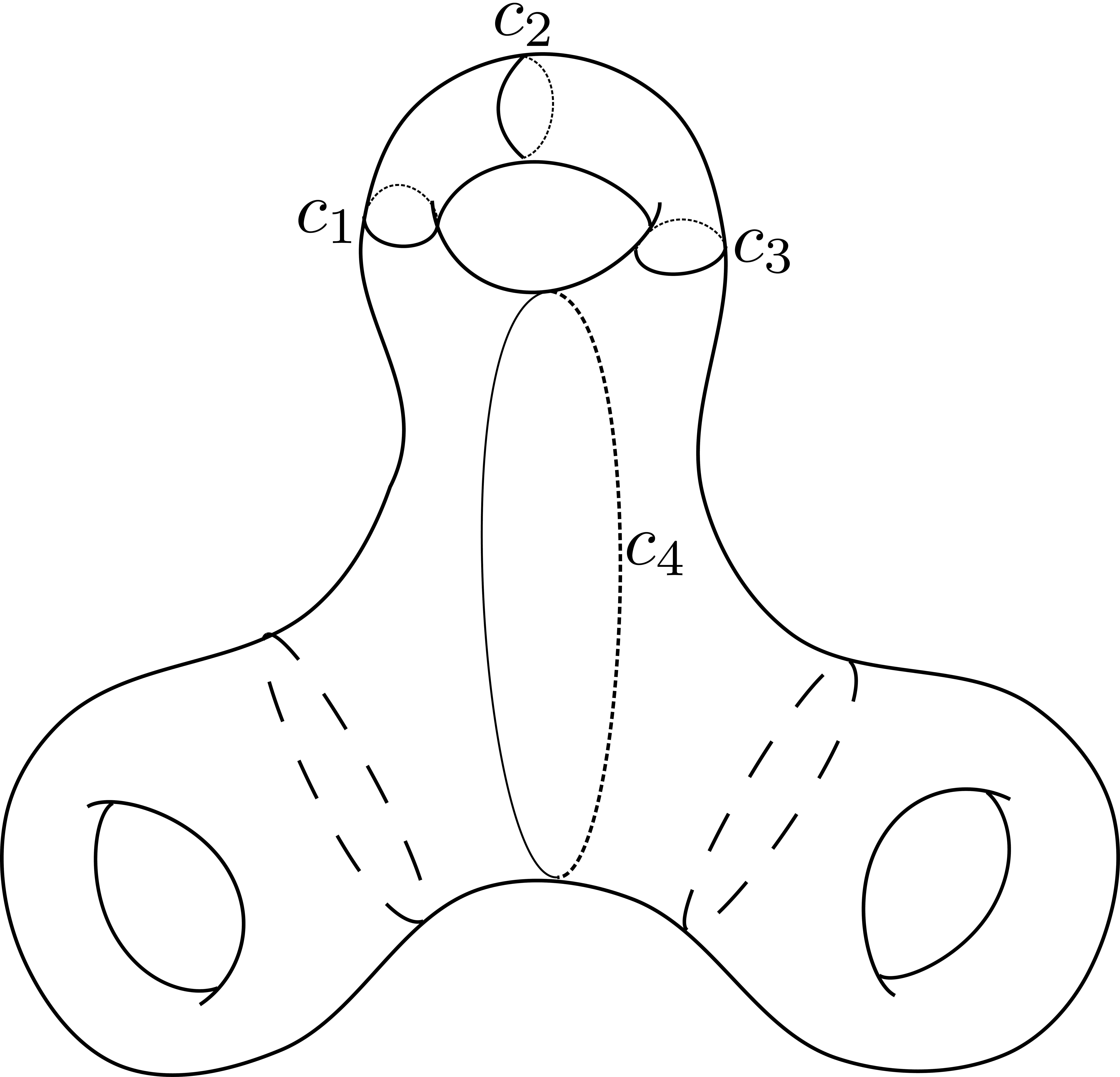}
	\includegraphics[scale=0.046]{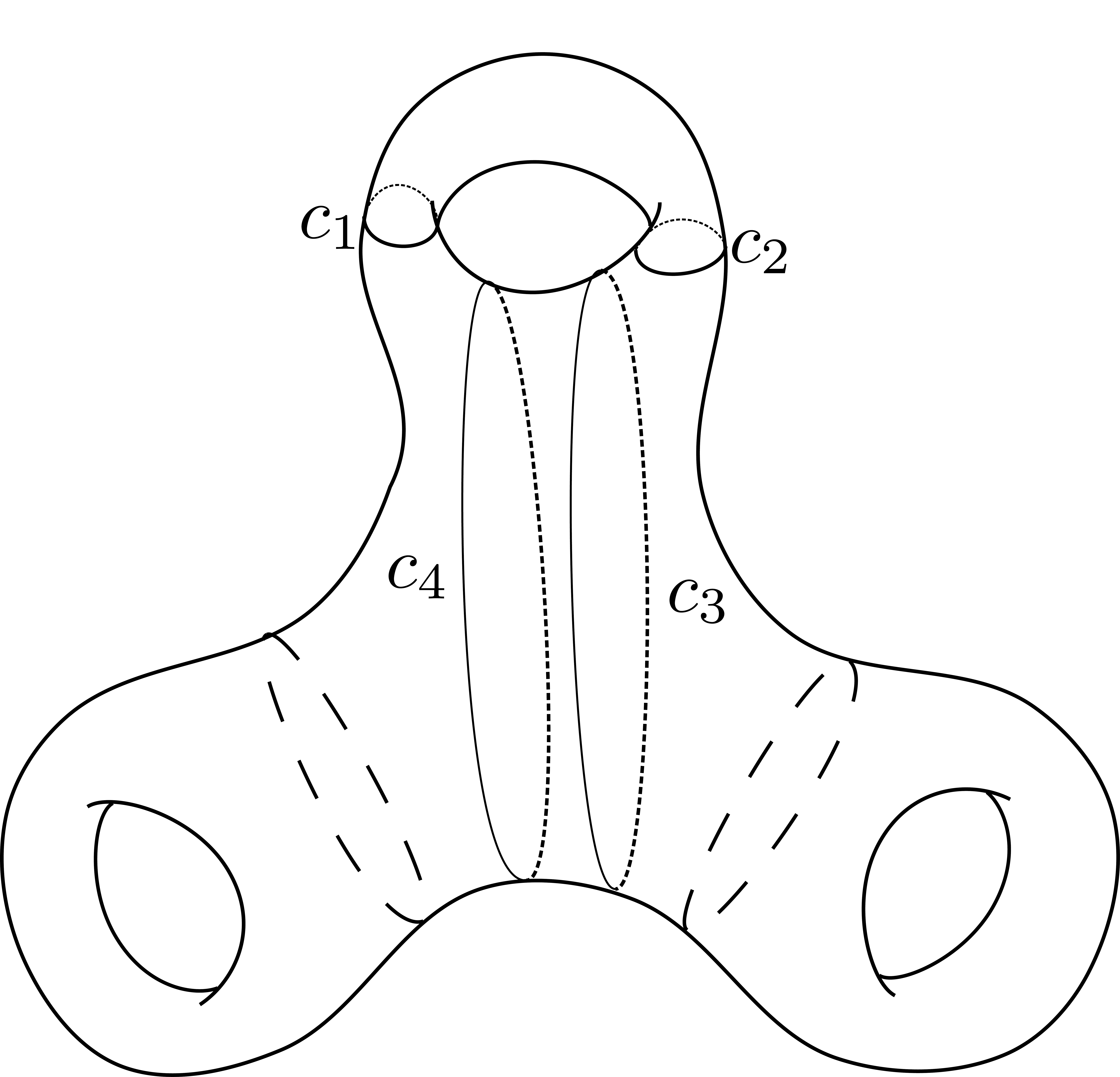}
	\caption{The possible components of $G-S$ (each dashed line represents a set of $t$-disks, possibly empty, but at least one nonempty).}	
	\label{figure:components}
\end{figure}

For the chosen $\tau$, consider a complete system of meridian disks $E=E_1\cup E_2\cup E_3$ of $H$, and assume that the number of intersections of $E$ with $S$, $|E\cap S|$, is minimal among all choices of $E$. 
Note that $E\cap S$ is non-empty, as, otherwise, the closure of $S-S\cap G$ would be a properly embedded essential punctured disk in the ball (closure of) $H-E$, which for fundamental group reasons is impossible. 
Furthermore, no component of $E\cap S$ is a closed curve. Otherwise, considering an innermost one in $E$, we obtain a compressing disk for $S-L$ or it also bounds a disk in $S$ 
and by cutting and pasting $E$ along this disk we can reduce $|E\cap S|$, contradicting its minimality.

Let $\alpha$ be an arc of $E\cap S$ in $E$.  Let $O$ be a disk cut by $\alpha$ from $E$, $\beta$ the arc $\partial O - \alpha$, which lies in $\partial G$, and $\omega_1$ and $\omega_2$ the disk components of $G\cap S$ containing the ends of $\alpha$. (See Figure \ref{figure:meridian_disk}.) If $\omega_1=\omega_2$, then $\alpha$ is called a {\em loop}.\\

\begin{figure}[ht]
	\includegraphics[scale=.07]{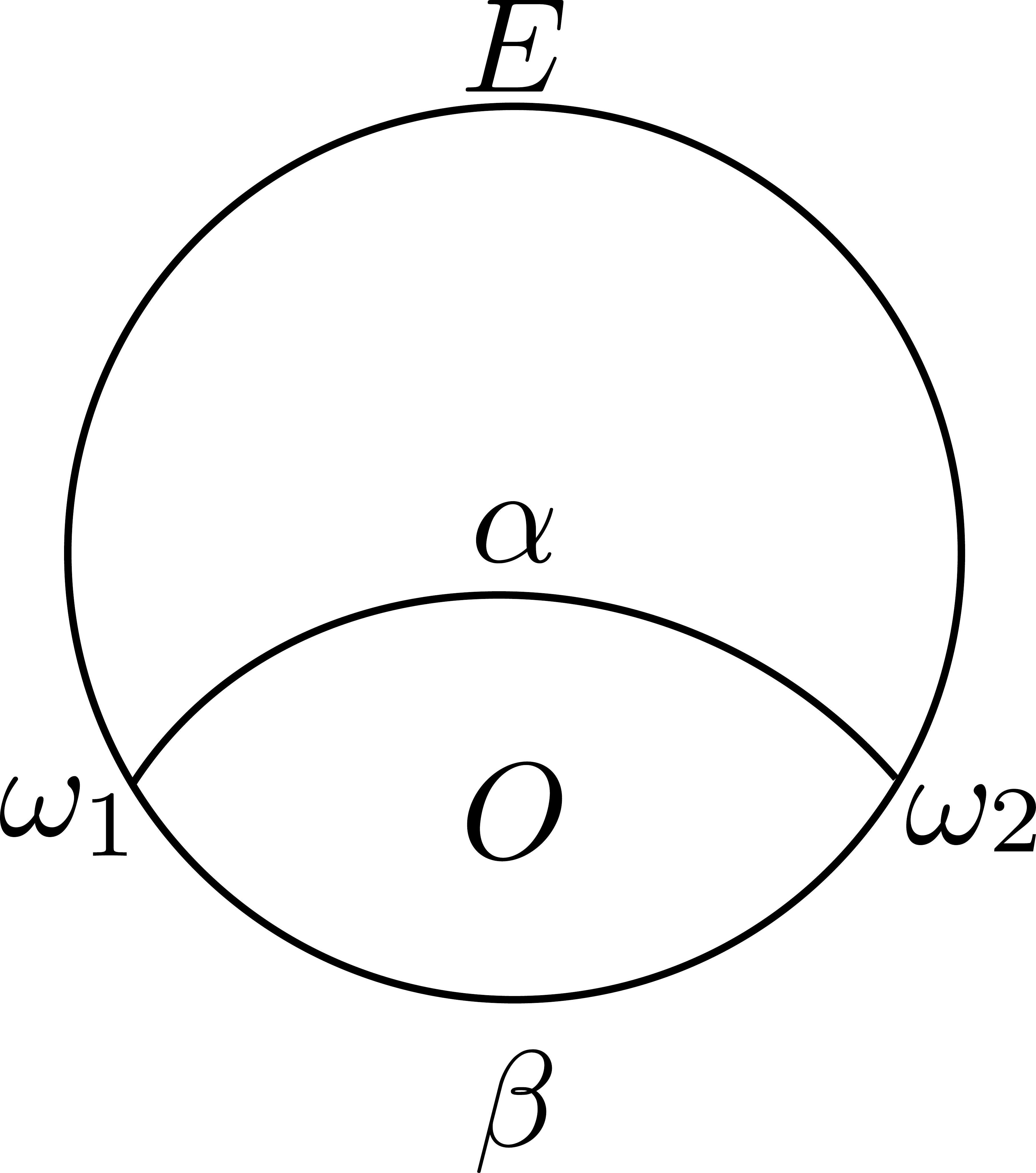}
	\caption{A component of $E\cap S$.}
	\label{figure:meridian_disk}
\end{figure}

\begin{lemma}\label{claim:looped_disks} For every $t$-disk $t_i$, an outermost arc of $E\cap S$ in $E$ among those with an end at $t_i$ is a loop. 
\end{lemma}
\begin{proof}  Let $\alpha$ be an outermost arc among those with one end at $\omega_1=t_i$ and $O$ the corresponding outermost subdisk of $E$. Since $\alpha$ is parallel to the boundary of $G$, then $G\cup N(\alpha)$ defines a stabilization of the Heegaard decomposition induced by the boundary of $G$. If $\omega_2\neq\omega_1$, then $\partial O$ only intersects $\omega_1$ once. Therefore the boundary of $\omega_1$ is primitive in $H-N(\alpha)$. Then, we can perform a destabilization of the genus-4 Heegaard decomposition defined by $G\cup N(\alpha)$, which corresponds to cutting $G\cup N(\alpha)$ along $\omega_1$, obtaining again a genus 3 Heegaard surface. This procedure reduces $|\tau\cap S|$, contradicting its minimality (this process corresponds to perform an isotopy that sends a neighborhood of $t_i$ through $O$ to a neighborhood of $\alpha$).
\end{proof}

Consider on $S$ the graph $\Gamma=(G\cap S,E\cap S)$ and the subgraph $\Gamma_c$ induced by the $c$-disks.

	\begin{lemma}\label{claim:loopless_disks} The graph $\Gamma_c$ has at least two connected components each with a $c$-disk without loops. 
\end{lemma}
\begin{proof}
From Lemmas \ref{tau and S} and \ref{claim:looped_disks}, there is a loop $\alpha$ at a $t$-disk $\omega$ of $G\cap S$ which separates $S$ into two components. If in one of these components every disk has loops, then there is an innermost loop. This loop and its corresponding disk of $G\cap S$ bound a disk in $S$, that can be used to reduce $|E\cap S|$ by cutting and pasting $E$ along this disk, which contradicts its minimality. Therefore in each component of $S-\omega-\alpha$ there is a disk without loops. By Lemma \ref{claim:looped_disks}, these disks are $c$-disks $c_i$ and $c_j$. Since $c_i$ and $c_j$ are separated by $\alpha\cup\omega$, and $\omega$ is a $t$-disk, they are in different connected  components of $\Gamma_c$.
\end{proof}

\begin{figure}[ht]
\includegraphics[scale=.07]{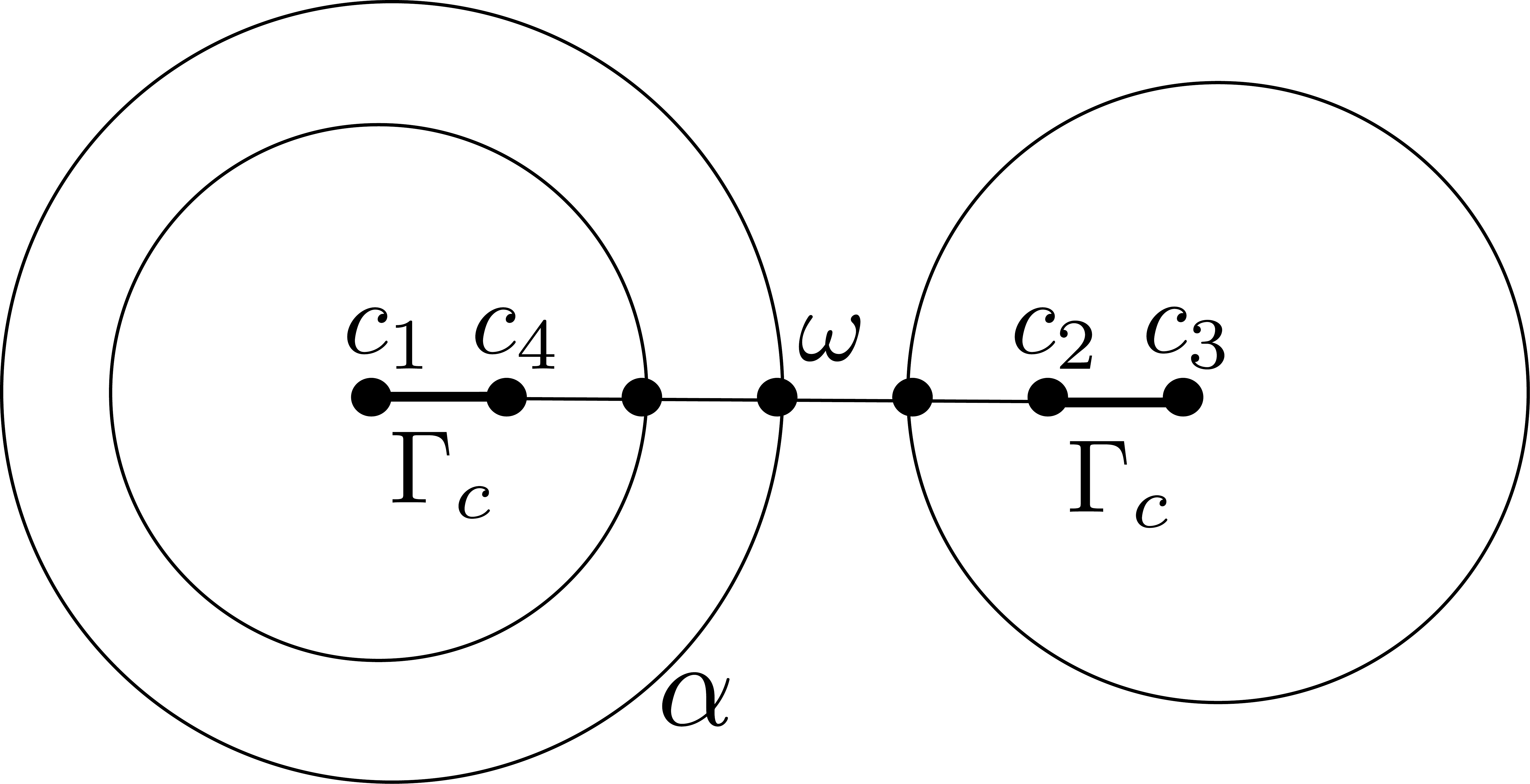}
		\caption{The graph $\Gamma$ and the subgraph $\Gamma_c$.}	
	\label{figure:graph}
\end{figure}

\

Suppose now that $\alpha$ is an outermost arc of $E\cap S$ in $E$, that is, the corresponding outermost disk $O$ is disjoint from $S$, and let $Q$ be the component of $G-S$ that contains $\beta$.

\begin{lemma}\label{claim:not_ball}$Q$ is not a ball.
\end{lemma}
\begin{proof}
	Suppose that $Q$ is a ball.
	If $\omega_1\neq \omega_2$, then $\omega_1$ and $\omega_2$ are $c$-disks, by Lemma \ref{claim:looped_disks}. Since $Q$ doesn't have more than two $c$-disks (see Figure \ref{figure:components}), $\omega_1$ and $\omega_2$ are connected by the arc $Q\cap L_1$, which is trivial in $Q$ and hence parallel to $\beta$. Hence, $Q\cap L_1$ being parallel to $\beta$ in $Q$ together with the disk $O$ being co-bounded by $\beta$ and $\alpha\subset S$, we have that $Q\cap L_1$ co-bounds a disk with $\alpha$ in the exterior of $S\cup L$. Then $S$ is boundary parallel in the exterior of $L$. This contradicts $S$ being essential in the exterior of $L$.\\ 
	Suppose now that $\omega_1=\omega_2$. Let $\Delta$ be a disk in $\partial Q$ with minimal intersection with $L$, bounded by $\beta$ and a subarc $\delta$ of $\omega_1$. Since $Q$ doesn't have more than two $c$-disks, $O\cup\Delta$ is a disk which intersects $L$ at most one time. Then $\alpha\cup\delta$ bounds a disk $A$ in $S$ containing at most one $c$-disk. By an isotopy on the ball bounded by $A\cup O\cup\Delta$, we can reduce $|E\cap S|$ (and possibly $|\tau\cap S|$), which contradicts their minimality. Note that if this ball contains an arc of $L_1$, then this arc is trivial in the ball and has one end in $\Delta$ and one end in $A$, so we can do isotopy in the exterior of $L$.
\end{proof}

By Lemma \ref{claim:not_ball}, $Q$ is either a solid torus or it has genus 2.\\

\noindent {\bf  Case 1:} Assume $Q$ is a solid torus.

\begin{lemma}\label{lemma:distinct c-disks}
$\omega_1$ and $\omega_2$ are distinct $c$-disks.
\end{lemma}
\begin{proof}Suppose that $\omega_1=\omega_2$. As $G$ has genus three and $S$ is disjoint from $L_2\cup L_3$, a solid torus component of $G-S$ contains either $L_2$ or $L_3$. Without loss of generality, suppose that $L_2$ is in $Q$. Let $A$ be a disk in $S$ cut by $\alpha\cup \omega_1$. The disk $A\cup O$ is properly embedded in the exterior of the solid torus $Q$. As $S^3$ has no lens space or $S^2\times S^1$ summand we have that  $A\cup O$ has boundary parallel to $L_2$. But as $A\cup O$ is disjoint from $L_3$, we have a contradiction to $L_2$ being linked with $L_3$. Hence, $\omega_1\neq\omega_2$ and, by Lemma \ref{claim:looped_disks}, $\omega_1$ and $\omega_2$ are $c$-disks.
\end{proof}

In this situation, the $c$-disks in $G$ are either all parallel, or only three are parallel, or consist of two pairs of two parallel disks in $G$, and all outermost arcs have, say, one end in $c_1$ and one end in $c_4$, as illustrated in Figure \ref{figure:case1b}. 

\begin{figure}[ht]
	\includegraphics[scale=0.046]{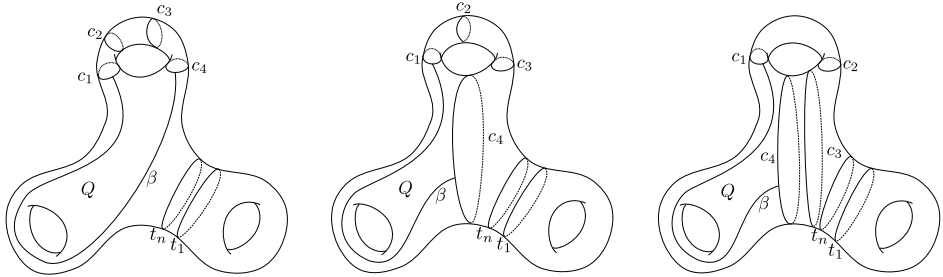}
	\caption{The disks of intersection of $G$ and $S$ when $Q$ is a solid torus.}	
	\label{figure:case1b}
\end{figure}

\begin{lemma}\label{disk_c2143}
An outermost arc among those with an end at $c_3$ is also an outermost arc among those with an end at $c_2$.
	\end{lemma}
	\begin{proof}
		Consider an outermost arc $\alpha$ of $E\cap S$ among those with an end at $c_3$,  the corresponding arc $\beta$ of $\partial E$ and the corresponding disk $O$ in $E$. If there is some arc with an end at a $t$-disk  in $O$ (including $\alpha$), then there is an outermost one, which is a loop, by Lemma \ref{claim:looped_disks}. Therefore, in its corresponding outermost disk $\Delta$, there is at least an outermost arc which, by Lemma \ref{lemma:distinct c-disks}, has ends in distinct $c$-disks, assumed to be $c_1$ and $c_4$, and, as $c_3$ is parallel to $c_1$ or $c_4$ in $G$, we have that $\Delta$ intersects $c_3$, which implies that $c_3$ is in the interior of $\beta$. This contradicts $\alpha$ being outermost among arcs of $E\cap S$ with an end at $c_3$.
		
		 Hence, $\alpha$ has the other end in a $c$-disk, and $\beta$ meets $S$ only at $c$-disks. Since $c_1$ and $c_4$ are in the same component of $\Gamma_c$, then, by Lemma \ref{claim:loopless_disks}, at least one of $c_2$, $c_3$ has no loops and is in a different component of $\Gamma_c$. Then, either $c_3$ is in the same component of $c_1\cup c_4$, or it is in the other component with $c_2$.
		Suppose, by contradiction, that in the interior of $O$ there is an arc $\alpha_2$ with an end at $c_2$.  As there are no $t$-disks in the interior of $O$, the other end of $\alpha_2$ is a $c$-disk. Therefore, $c_2$ and $c_3$ are  in the same component of $\Gamma_c$. Then, in $\Gamma_c$, $c_2$ and  $c_3$ can only be connected to each other, and the other end of $\alpha_2$ is $c_3$, a contradiction with $\alpha$ being outermost.

		Therefore, in the interior of $O$ all arcs have ends at $c_1$ or $c_4$. A intersection consecutive to a $c_1$ is a $c_2$ so it must be the other end of $\alpha$.
\end{proof}

Suppose that the $c$-disks are parallel two by two in $G$. As $Q$ is outside $S$, the ball $B$ contains the two 1-handles cut by $S$ from $G$, connecting $c_1$ to $c_2$ and $c_3$ to $c_4$, $c_{1;2}$ and $c_{3;4}$ respectively. In the graph $\Gamma$, $c_1$ is connected to $c_4$ by $\alpha$, and, from Lemma \ref{disk_c2143}, $c_2$ is connected to $c_3$. As $L_2$ and $L_3$ are in the same side of $S$ and for the $c$-disks being parallel two by two in $G$, the number $n$ of $t$-disks is even. By Lemma \ref{claim:looped_disks}, the outermost arc among those with ends in a specific $t$-disk has both ends in the same disk. At least a sequence of these arcs is parallel in $E$. (See Figure \ref{figure:outermost_tn}.) Using the disks making these arcs parallel in $E$, and considering an outermost one in $B$, we have that either there is an essential torus in the exterior of $G$, a contradiction with it being a handlebody; or a 1-handle in $G$ cut by two consecutive $t$-disks is parallel to $S$ in the exterior of $G$, and we can reduce $|S\cap \tau|$, contradicting its minimality. (See also the argument for the proof of Lemma \ref{lemma:same t-disk}.)

Suppose now that at least three $c$-disks are parallel in $G$. Consider an arc $\gamma$ of $E\cap S$ with an end at $c_3$. In at least one of the disks $O$ cut off by $\gamma$ there is an arc with an end at a $t$-disk, which implies that there is another arc with an end at $c_3$. Since the number of arcs with an end at $c_3$ is finite, we can assume that every arc in $O$ with an end at $c_3$ is outermost. By Lemma \ref{disk_c2143}, every such arc has the other end at $c_2$. Therefore, there is a disk $\Delta$ cut off from $O$ by a sequence of arcs with ends at $c_2$ and $c_3$, as in Figure \ref{figure:Scharlemann}. Let $c_{2;3}$ be the component of $G-S$ whose boundary contains $c_2\cup c_3$, which is a ball, and let $B$ be the ball cut from $S$ not containing $c_{2;3}$. As $c_3$ is a meridian of the solid torus obtained by gluing $B$ and $c_{2;3}$ along $c_2$ and $c_3$ and $\Delta$ crosses $c_3$ more than once always in the same direction (such a disk is known as a Scharlemann cycle), we obtain a contradiction with $S^3$ having no lens space summands.

\begin{figure}[ht]
	\includegraphics[scale=0.046]{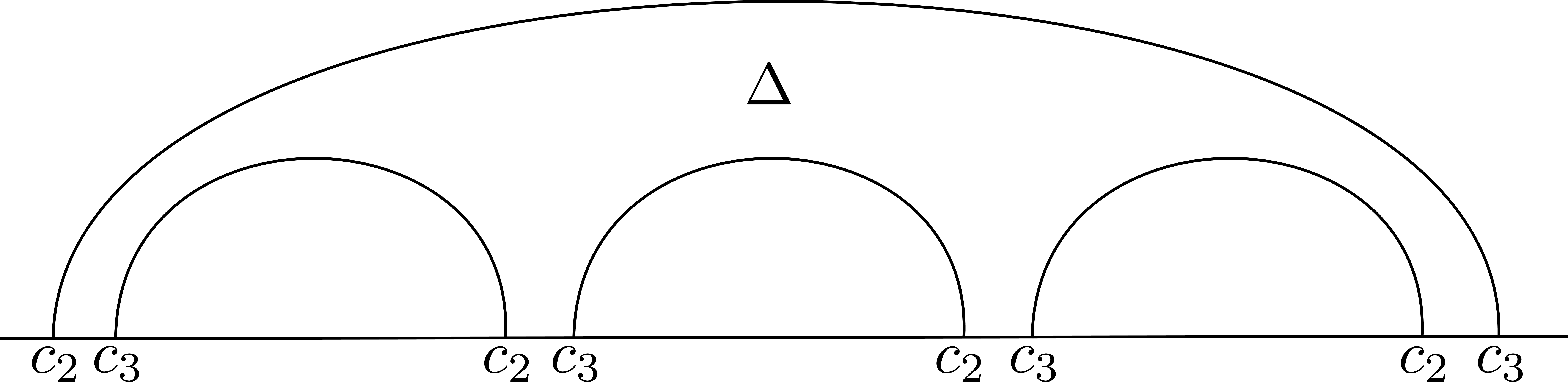}
	\caption{A Scharlemann cycle $\Delta$.}	
	\label{figure:Scharlemann}
\end{figure}

\

\noindent {\bf  Case 2:} Assume $Q$ has genus 2.

\begin{lemma} \label{lemma:same t-disk}
$\omega_1$ and $\omega_2$ are the same $t$-disk.
\end{lemma}
\begin{proof}
	Since $Q$ has genus two and $G$ has genus three, $Q$ is cut from $G$ by one or two disks. If $Q$ is cut by one disk, it must be a $t$-disk, since no $c$-disk is separating in $G$. Therefore, $\omega_1$ and $\omega_2$ are the same $t$-disk. If $Q$ is cut from $G$ by two disks, then all disks of $G\cap S$ are parallel in $G$. As a $t$-disk cannot be parallel to a $c$-disk, we have that there is no $t$-disk in $G\cap S$, which contradicts Lemma \ref{tau and S}.
\end{proof}

Let $t_1, t_2, \ldots, t_n$ be the parallel $t$-disks in $G$ numbered in consecutive order with both ends of $\beta$ at $t_1$, as in Figure \ref{figure:case1a}.

\begin{figure}[ht]
	\includegraphics[scale=0.046]{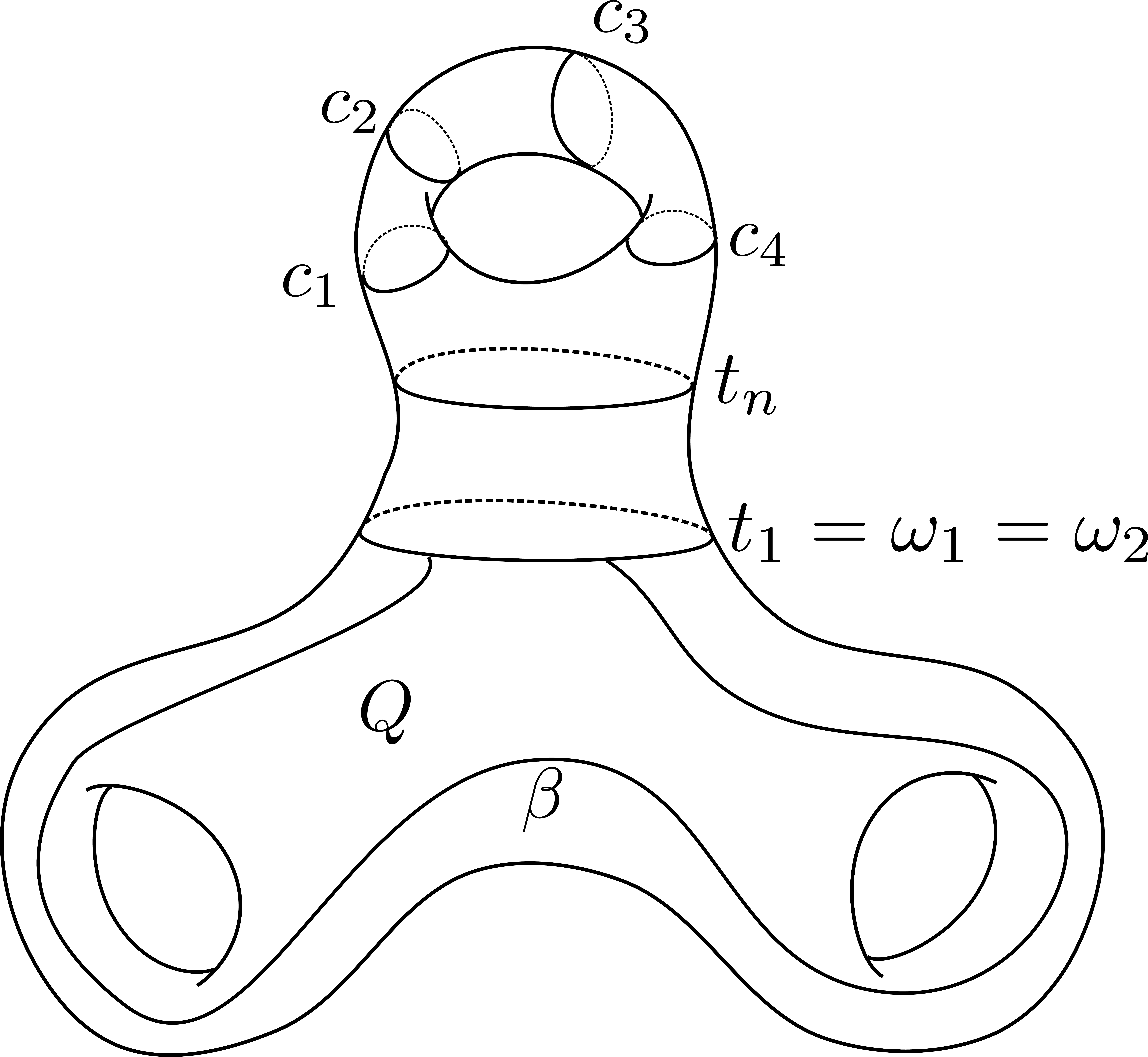}
	\caption{The disks of intersection of $G$ and $S$ when $Q$ has genus 2.}	
	\label{figure:case1a}
\end{figure}

From the number of components $|E\cap S|$ being finite, there is a sequence of arcs of $E\cap S$ in $E$, as in Figure \ref{figure:outermost_tn}, denoted by $\alpha_i$ when it has both ends in $t_i$.\\

\begin{figure}[ht]
	\includegraphics[scale=0.046]{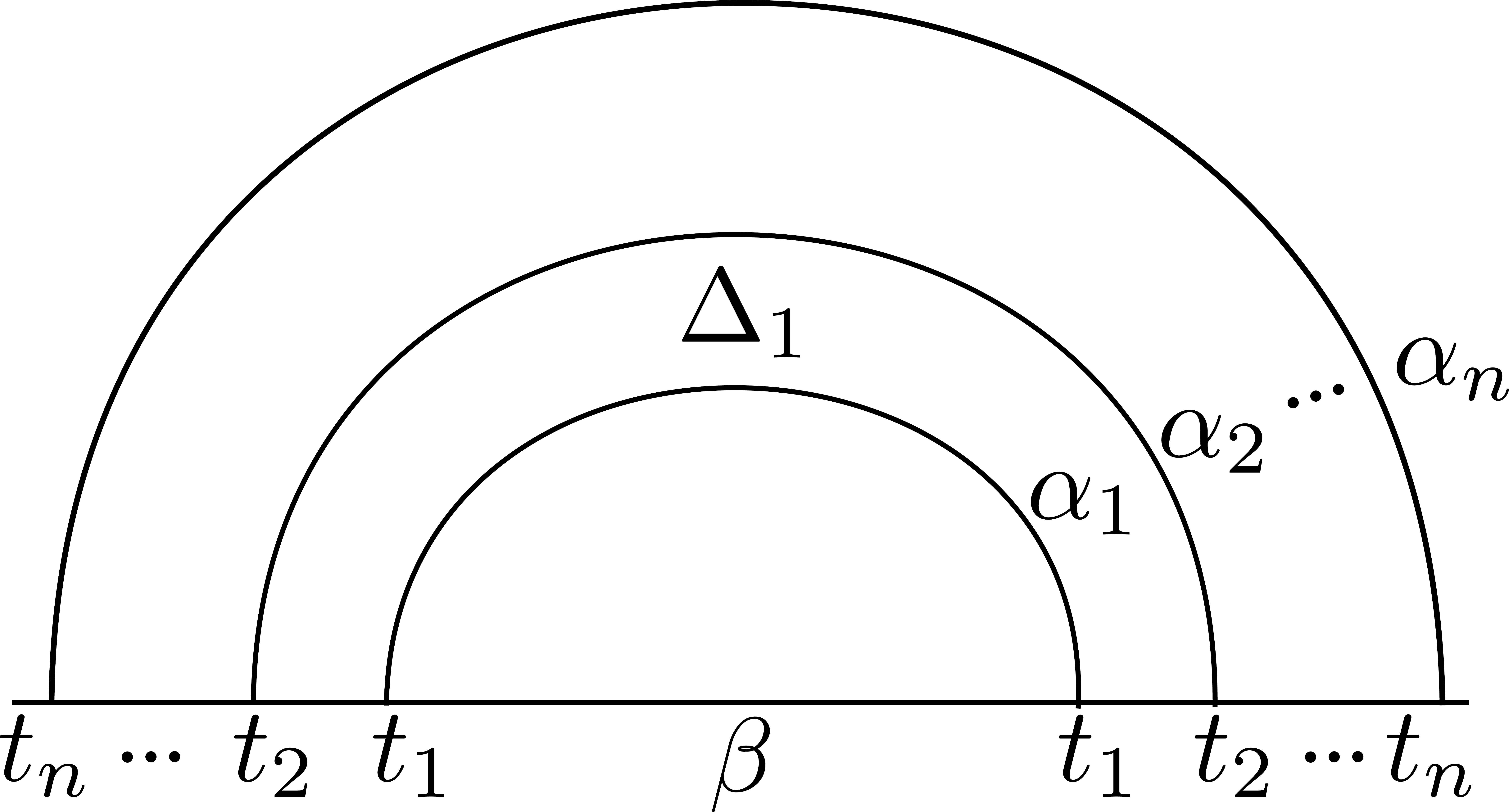}\qquad\qquad\qquad 
	\includegraphics[scale=0.1]{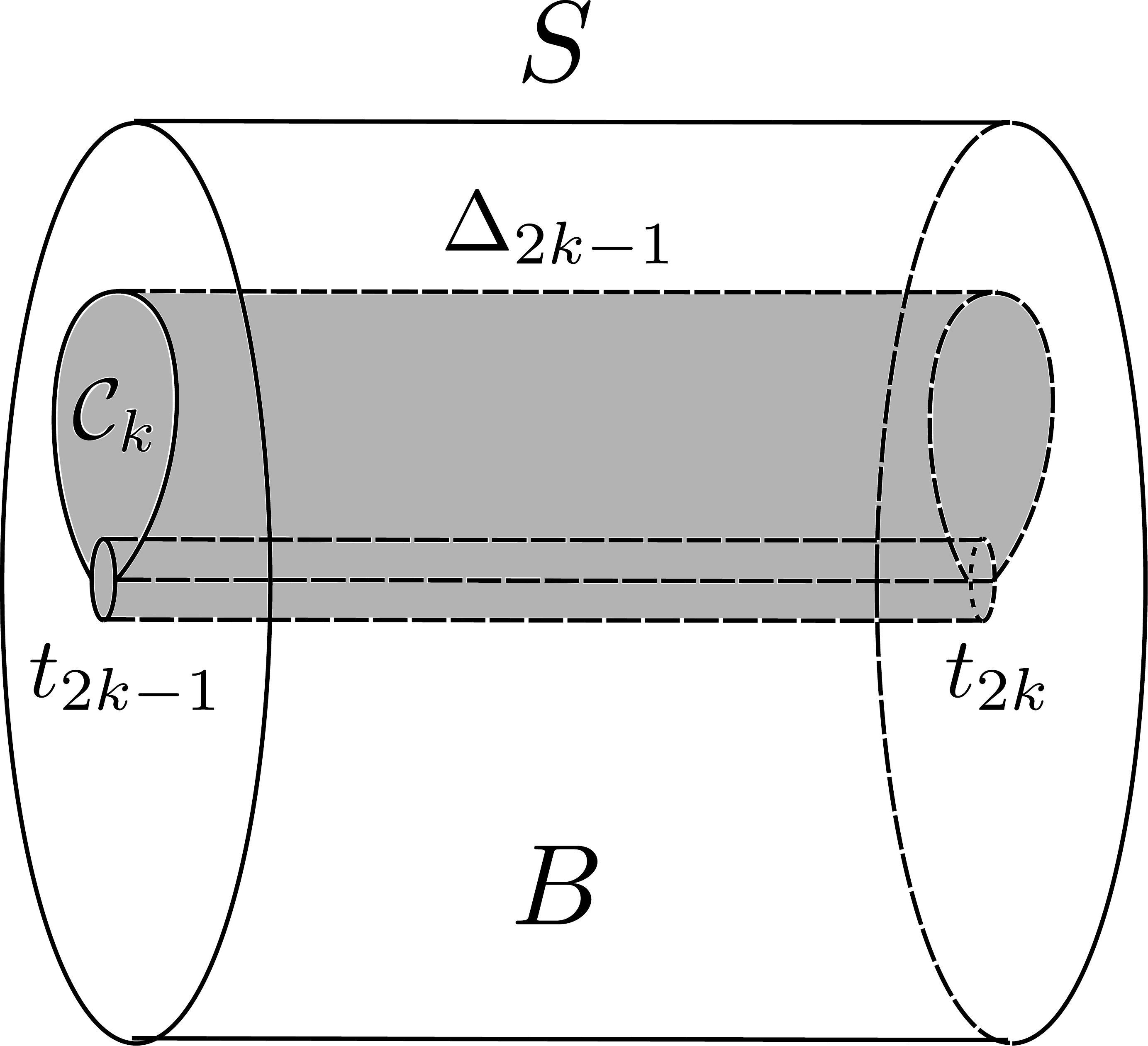}
	\caption{A sequence of arcs $\alpha_1,\ldots,\alpha_n$ of $E\cap S$ in $E$ (left); The $1$-handle ${\mathcal C}_k$ (right) .
	}	
	\label{figure:outermost_tn}
\end{figure}

\begin{lemma}\label{lemma:odd}
The number $n$ of $t$-disks in $G$ is odd.
\end{lemma}
\begin{proof}	
Suppose $n$ is even. Let $\Delta_{2k-1}$ be the disk cut by $\alpha_{2k-1}\cup \alpha_{2k}$ from $E$, and $t_{2k-1; 2k}$ the $1$-handle cut by $t_{2k-1}\cup t_{2k}$ from $G$. Then, the 1-handle that $t_{2k-1; 2k}\cup \Delta_{2k-1}$ cuts from $B$ together with $t_{2k-1; 2k}$ defines a 1-handle ${\mathcal C}_k$  in $B$, which  intersects $S$ in two disks. Both disks intersect $L_1$, because otherwise we could reduce $|E\cap S|$. Suppose that one ${\mathcal C}_k$ contains both arcs of $L_1$ inside $S$. Then, all ${\mathcal C}_k$'s with this property are nested in $B$, as all ${\mathcal C}_k$'s intersect $L_1$ and their annuli in the interior of $B$ are pairwise disjoint (See Figure \ref{figure:Cknested}.)

\begin{figure}[ht]
	\includegraphics[scale=0.15]{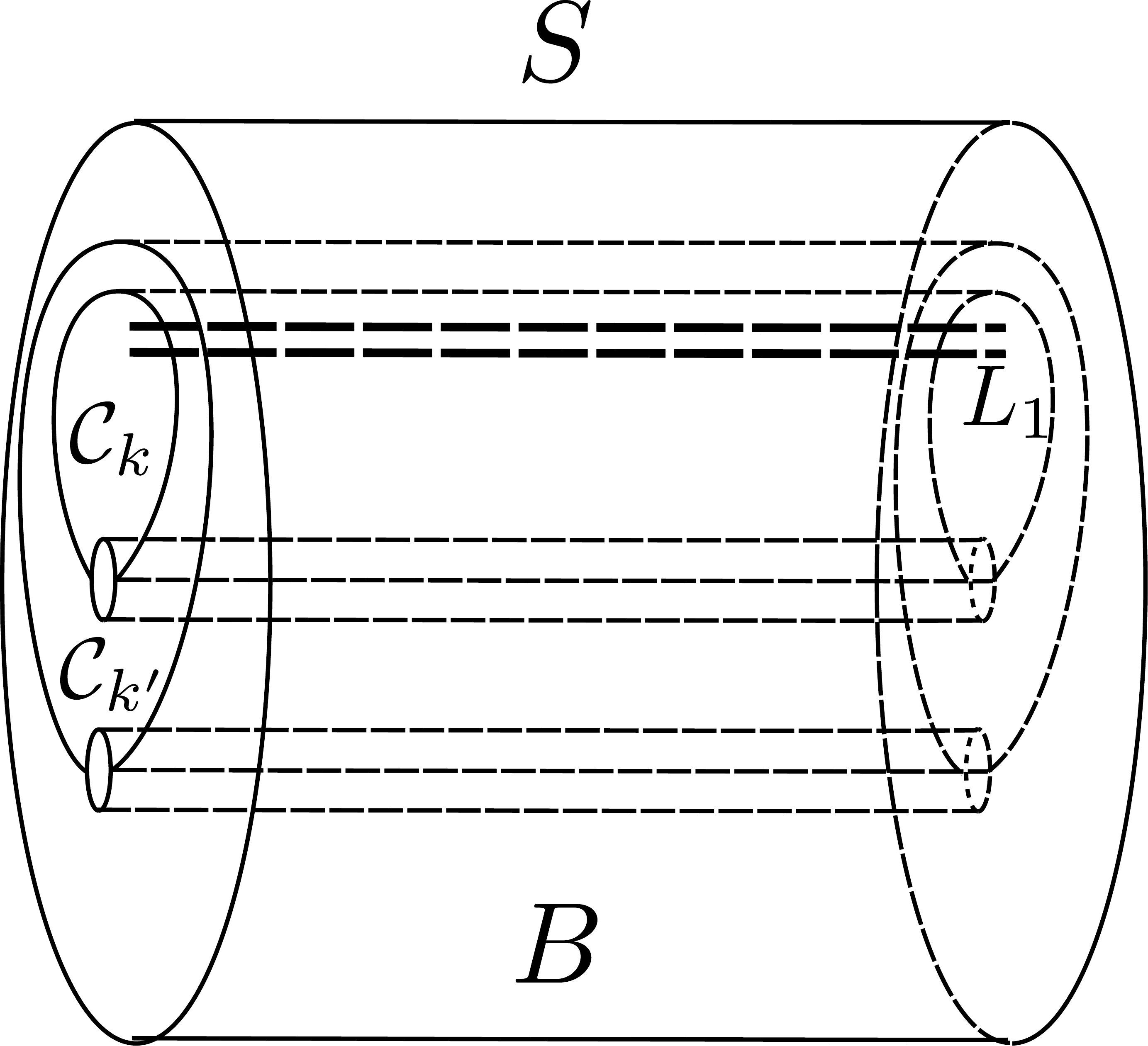}
	\caption{The nested 1-handles  ${\mathcal C}_k$'s.}	
	\label{figure:Cknested}
\end{figure}

By tubing along an outermost one inside $S$, we obtain a torus in $H$. Either this torus is trivial in $B$ and we can reduce $|G\cap S|$, a contradiction, or it is an essential torus and, therefore, there is an essential torus in $H$, a contradiction. Therefore, every ${\mathcal C}_k$ contains only one arc of $L_1$. Since $L_1$ is unknotted, this arc is trivial in ${\mathcal C}_k$ and parallel to $t_{2k-1; 2k}$. Then, inside $S$, $G$ is defined by a collection of parallel 1-handles with the pattern of $K_0$, and therefore it contains an essential torus in $H$, a contradiction. Hence, $n$ is odd.
\end{proof}

Note that the strings from $c_1$ to $c_4$ and $c_2$ to $c_3$ are inside of $S$ and are knotted. Let $\gamma$ be an outermost arc of $E\cap S$ in $E$ after the arcs $\alpha_n$. (See Figure \ref{figure:S2} (left).)

\begin{figure}[ht]
	\includegraphics[scale=0.06]{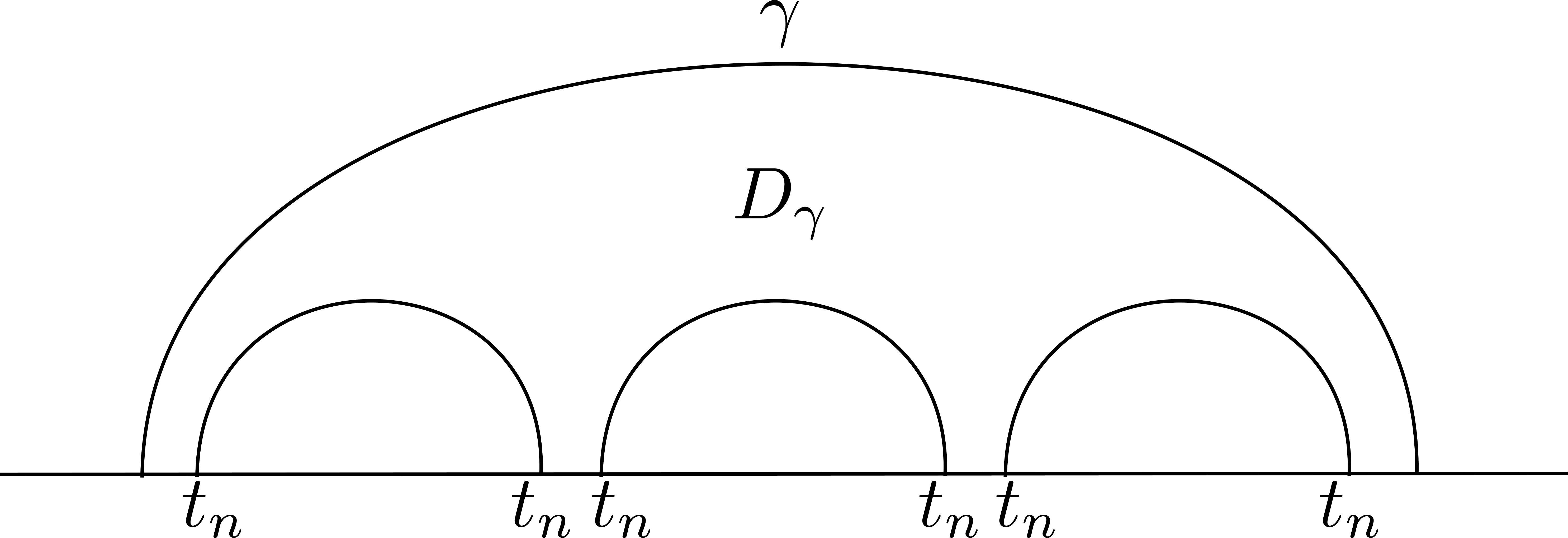}
	\qquad
	\includegraphics[scale=0.025]{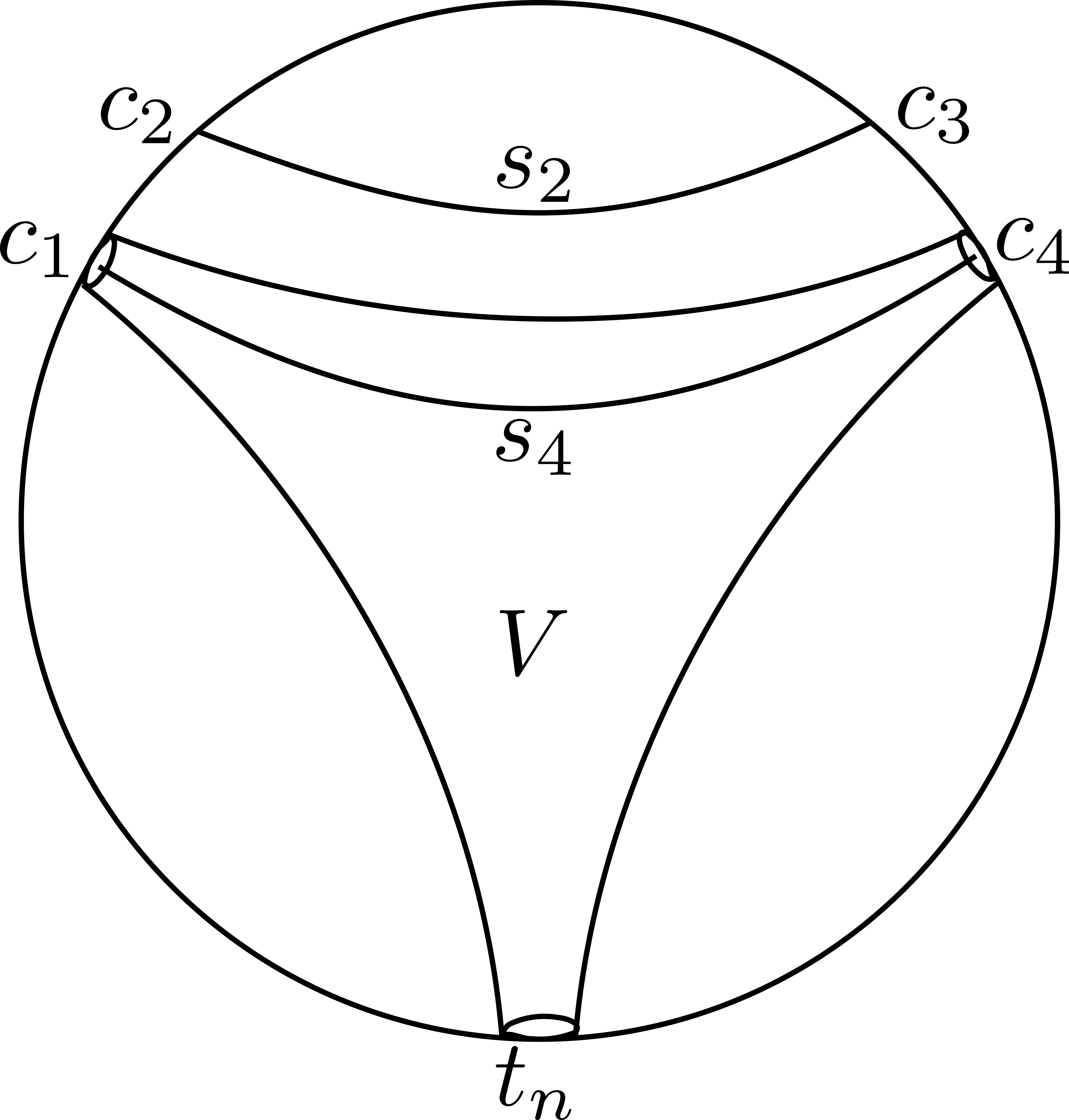}
	\caption{An outermost arc $\gamma$ of $E\cap S$ after the arcs $\alpha_n$ (left); the component $V$ cut from $G$ by $t_n\cup c_1\cup c_4$ (right).}	
	\label{figure:S2}
\end{figure}

\begin{lemma}[cf., Lemma 3.4 of \cite{Nogueira}]
\label{lemma:both_ends_in_c1_or_both_ends_in_c4}
The arc $\gamma$ has either both ends in $c_1$ or both ends in $c_4$. 
\end{lemma}
\begin{proof} Denote by $D_\gamma$ the disk cut by $\gamma$ and the arcs $\alpha_n$ and by $V$ the component cut from $G$ by $t_n\cup c_1\cup c_4$. By Lemma \ref{lemma:odd}, $n$ is odd, $V$ is in $B$. (See Figure \ref{figure:S2} (right)). Let $T$ be the solid torus defined by identifying $B'$ to $V$ along $c_1\cup t_n$. We also denote the arc of $B\cap L_1$ in $V$ by $s_4$, which has one end in $c_1$ and one end in $c_4$, and the other arc of $B\cap L_1$ by $s_2$, which has one end in $c_2$ and the other end in $c_3$. We will show that, if $\gamma$ has an end in $t_n$ or one end in $c_1$ and the other end in $c_4$, then one of the arcs $s_2$ or $s_4$ is unknotted in $B$, which contradicts it having a non-trivial knot pattern in $B$.
\medskip\\
Assume first that $\gamma$ has both ends in $t_n$. As $S^3$ has no lens space or $S^1\times S^2$ summands, we have that $\partial D_\gamma$ is inessential in $T$. Let $D_T$ denote the disk bounded by $\partial D_\gamma$ in $\partial T$. The disk $D_T$ intersects $L_1$ in at least two components. In fact, consider the arcs $D_T\cap \partial (c_1\cup t_n)$. Then, there are at least two outermost disks cut by these arcs in $D_T$. Let $O$ be one of these disks. In case $O$ is disjoint from $L_1$, then either it contains some disk $S\cap \tau$ and using an innermost loop of $E\cap S$ attached to these disks in we can reduce $|E\cap S|$, or the interior of $O$ is disjoint from $G$  and by cutting and pasting along $O$, or by an isotopy of this disk from $V$ into $S$, we can also reduce $|E\cap S|$, contradicting its minimality.\\
Let $R$ be the ball bounded by $D_\gamma\cup D_T$. Then, the ball $R$ contains either $s_2$, or a segment of $s_4$, or a union of these two.\\ 
Suppose $R$ contains the string $s_2$ only. As there are no local knots in the tangle $(B;B\cap L)$, the arc $s_2$ in $R$ is trivial. As both ends of $s_2$ are in $D_T$, then $s_2$ is parallel to $D_T$. Hence, as we can push $D_T$ to $S$ from $V$, we have that $s_2$ is unknotted in $B$.\\
Suppose now that $R$ contains also a segment of the string $s_4$. As $R$ intersects each component of $B\cap L$ at a single arc, and the exterior of $R$, denoted $R'$, contains the tangle $(B', B'\cap L)$, we have that $(R', R'\cap L)$ is essential. As $|\partial R\cap G|<|S\cap G|$ ($t_n$ is not in $D_T$), if the tangle $(R, R\cap L)$ is essential, we have a contradiction to the minimality of $|S\cap G|$. Therefore, and as there are no local knots, $(R, R\cap L)$ is a trivial tangle. Then, the string $s_2$ is unknotted in $R$ and, as before, also unknotted in $B$.\\
Finally, suppose that $R$ contains only a segment of the string $s_4$. As $D_T$ is disjoint from $c_1$, if we cut $V$ along $c_1$, we obtain a ball $B_V$ intersecting $s_2$ at a single component. As $B_V\cup R$ is $\partial$-parallel in $B$, the exterior of $B_V\cup R$ in $B$, denoted $B_R$, is a ball in $B$ intersecting $B\cap L$ at $s_2$ and a segment of $s_4$. As $|\partial B_R\cap G|<|S\cap G|$, following a similar reasoning as when $R$ contains two arcs, we also have that $s_2$ is unknotted in $B_R$. As $B_R\cup R$ is $\partial$-parallel in $B$, the 2-sphere $\partial B_R$ is isotopic to $S$ relative to $B_R\cap S$ in $B$. Then, $s_2$ is also unknotted in $B$.
\medskip\\
Assume now that $\gamma$ has one end in $t_n$ and, without loss of generality, the other end in $c_1$. We isotope $S$ in $S^3$ through $L$ along a regular neighborhood of a disk in $V$ which intersects $s_4$ once, intersects the disk $t_n$ along a single arc and separates $c_1$ and $c_4$ in $V$. Along this disk we also separate the disk $t_n$ into the disks $t_{n1}$ and $t_{n4}$, and $V$ into two 1-handles: $V_1$ connecting $c_1$ and $t_{n1}$ and $V_4$ connecting $c_4$ and $t_{n4}$. Let $S^*$ denote the sphere obtained after the isotopy of $S$. After the isotopy of $S$, the boundary of $D_\gamma$ lies in $S^*$, in $\partial V_1$ and in $\partial V_4$. The arcs of $\partial D_\gamma \cap V_4$ have both ends attached in $t_{n4}$. Hence, we can isotope these arcs to $S^*$. Also, all but one arc of $\partial D_\gamma \cap V_1$ has both ends in $t_{n1}$, with the other arc being $\gamma$ which has one end in $c_1$ and the other in $t_{n1}$. We isotope all arcs of $\partial D_\gamma \cap (V_1\cup V_4)$ with both ends in $t_{n1}$ or both ends in $t_{n4}$ into $S^*$. We are left with the disk $D_\gamma$ with boundary defined by one arc in $S^*$ and the other arc in $\partial V_1$ with one end in $c_1$ and the other end in $c_4$. Using this disk, we can isotope $t_{n1}$ through $S^*$. After this isotopy, we obtain a sphere $S^*$ in $B$ bounding a ball $B^*$ intersecting $B\cap L$ in the whole string $s_4$ and in a segment of $s_2$. We also have $|S^*\cap G|<|S\cap G|$, as $t_n$ was eliminated from the intersection. So, as before, $(B^*, B^*\cap L)$ cannot be essential, which, as there are no local knots, implies that the tangle is trivial. Then, in particular, $s_4$ is unknotted in $B^*$ and as $S^*$ intersects $S$ in a disk, we also have that $s_4$ is unknotted in $B$.
\medskip\\
Finally, assume that $\gamma$ has one end in $c_1$ and the other end in $c_4$. Then each arc of $D_\gamma\cap S -\gamma$ co-bounds a disk in $S- t_n$, with $\partial t_n$,  disjoint from $c_1\cup c_4$. Hence, we can isotope the arcs $\partial D_\gamma \cap S -\gamma$ in this disk. So, after the isotopy, abusing the notation, $\partial D_\gamma$ is defined by $\gamma$ and an arc in $S$ from $c_1$ to $c_4$. As $s_4$ is trivial in $V$, it is parallel to $\gamma$ in $V$. Therefore, $s_4$ is unknotted in $B$.
\end{proof}

By Lemma \ref{lemma:both_ends_in_c1_or_both_ends_in_c4}, each arc $\gamma$ has both ends in $c_1$ or both ends in $c_4$. Consider an outermost arc $\delta$ with an end in $c_2$ and the corresponding outermost disk $O$. If, in $O$, there is an arc $\gamma$ with both ends in $c_1$, then the consecutive intersections must be $c_2$'s, hence $\delta$ has both ends in $c_2$. Similarly, if $\gamma$ has both ends in $c_4$, we have an arc with both ends in $c_3$, and the consecutive intersections must be $c_2$, hence $\delta$ has again both ends in $c_2$. This contradicts Lemma \ref{claim:loopless_disks}.

\vspace{.4cm}
\noindent
\address{\textsc{Department of Mathematics,\\
Universidade Federal do Cear\'a}}\\

\vspace{.4cm}
\noindent
\address{\textsc{CMUC, Department of Mathematics,\\
University of Coimbra}}\\
\email{\textit{E-mail:}\texttt{
		nogueira@mat.uc.pt}}

\vspace{.4cm}
\noindent
\address{\textsc{CMUC, Department of Mathematics,\\
		University of Coimbra}}\\
\email{\textit{E-mail:}\texttt{
		ams@mat.uc.pt}}

\end{document}